\documentclass[11pt,a4paper]{amsart}
\usepackage{amssymb, amsmath, amsthm}

\usepackage{tikz}
\usetikzlibrary{calc}
\usepackage{caption}
\usetikzlibrary{arrows}
\usepackage{pgfplots}
\pgfplotsset{compat=1.15}
\usepackage{mathrsfs}
\usepackage{pstricks}
\usepackage{mathrsfs}
\usepackage{extarrows}
\usepackage{amsthm}
\usepackage{amssymb}
\usepackage{xypic,enumerate, comment, enumitem}
\usepackage[all,cmtip]{xy}
\usepackage{fancyvrb}
\usepackage{longtable}
\usepackage{array}
\newcolumntype{L}{>{$}l<{$}} 
\allowdisplaybreaks

\usepackage{hyperref}
\hypersetup{
    colorlinks=true,       
    linkcolor=blue,        
    citecolor=blue,        
    filecolor=blue,        
    urlcolor=blue          
}

\theoremstyle{plain}

\newtheorem{theorem}{Theorem}[section]
\newtheorem{proposition}[theorem]{Proposition}
\newtheorem{corollary}[theorem]{Corollary}
\newtheorem{lemma}[theorem]{Lemma}
\newtheorem{conjecture}[theorem]{Conjecture}
\newtheorem*{theorem*}{Theorem}
\newtheorem*{proposition*}{Proposition}
\newtheorem*{corollary*}{Corollary}
\newtheorem*{lemma*}{Lemma}
\newtheorem*{conjecture*}{Conjecture}
\newtheorem*{thmA}{Theorem A}
\newtheorem*{thmB}{Theorem B}
\newtheorem*{thmC}{Theorem C}

\theoremstyle{definition}
\newtheorem{definition}[theorem]{Definition}
\newtheorem{example}[theorem]{Example}

\newtheorem*{definition*}{Definition}
\newtheorem*{example*}{Example}

\theoremstyle{remark}
\newtheorem{remark}[theorem]{Remark}
\newtheorem*{remark*}{Remark}

\newcommand{\G}{\operatorname{G}}

\newcommand{\bP}{\mathbb{P}}
\newcommand{\bL}{\mathbb{L}}

\newcommand{\cE}{\mathcal{E}}

\newcommand{\cF}{\mathcal{F}}
\newcommand{\cS}{\mathcal{S}}

\newcommand{\cG}{\mathcal{G}}

\newcommand{\cU}{\mathcal{U}}

\newcommand{\cX}{\mathcal{X}}

\newcommand{\cW}{\mathcal{W}}

\newcommand{\tF}{\tilde{F}}

\newcommand{\Pic}{\operatorname{Pic}}

\renewcommand\bar\overline
\renewcommand\tilde\widetilde

\newcommand{\Sym}{\operatorname{Sym}}
\renewcommand{\geq}{\geqslant}
\renewcommand{\leq}{\leqslant}

\newcommand{\PP}{\bP}
\newcommand{\OO}{\mathcal{O}}

\title{The Fermat cubic and monodromy of lines}
\author{Frank Gounelas}
\address{Georg-August-Universit\"at G\"ottingen, Fakult\"at f\"ur Mathematik und Informatik, Bunsenstr. 3-5, 37073
G\"ottingen, Germany}
\email{gounelas@mathematik.uni-goettingen.de}

\author{Alexis Kouvidakis}
\address{Dept. of Mathematics and Applied Mathematics, University of Crete, 70013 Heraklion, Greece.}
\email{kouvid@uoc.gr}

\date{\today}
\makeatletter
\@namedef{subjclassname@2020}{\textup{2020} Mathematics Subject Classification}
\makeatother

\begin{document}

\maketitle
\begin{abstract}
In this paper we study properties of the locus of second type lines of a general cubic threefold and fourfold. By
analysing the geometry of the Fano scheme of lines of the Fermat cubic fourfold and in particular giving an explicit
description of the locus of second type lines, we deduce that the Voisin map is birational over the second type locus.
For a general cubic threefold, by studying properties of the second type locus again, we compute that various natural
geometric monodromy groups are the full symmetric group.
\end{abstract}

\setcounter{tocdepth}{1}
\tableofcontents

\section{Introduction}\label{sec:introduction}

For a smooth cubic hypersurface $X\subset\PP^{n+1}$ one defines the Fano scheme of lines
$F(X)\subset\mathrm{G}(2,n+2)$ parametrising lines contained in $X$. This is a smooth and irreducible variety with rich
geometry and history, in particular in the low dimensional cases where, e.g., for $n=2$ we famously have that $F(X)$
consists of 27 smooth points, for $n=3$ we have that $\mathrm{Alb}(F(X))$ is isomorphic to the intermediate Jacobian
$J(X)$ used by Clemens--Griffiths to conclude that $X$ is not rational, and for $n=4$ we have that $F(X)$ is an
irreducible holomorphic symplectic variety deformation equivalent to $S^{[2]}$, the Hilbert scheme of two points on a K3
surface and conjectures abound relating the geometry of $F(X)$ and rationality of $X$.

One may also consider the locus $F_2(X)\subset F(X)$ of lines of second type, roughly consisting of lines with too many
tangent 2-planes. The universal families $\bL, \bL_2$ over $F(X), F_2(X)$ provide useful correspondences, relating
the Chow and cohomology theories of $X$ and $F(X), F_2(X)$. The geometry and properties of locus $F_2(X)$ (which is
non-empty only if $n\geq3$ and is smooth and irreducible if $X$ is general), though intimately tied to the geometry of
$X$, is rather complicated and we refer to \cite{huybrechts} for basic results and motivation. When $n=3$, Fano proved
that $F_2(X)$ is a bicanonical curve in the surface $F(X)$, which is smooth and irreducible and of genus 91 if the cubic
is general. In the case where $n=4$, only recently have its invariants become known (see \cite{huybrechts,
cubicfourfolds2}). 

The first aim of this paper is to obtain a description (see Proposition \ref{prop:fermatdegW}) of the locus $F_2(X)$ in
the case where $X$ is the Fermat cubic fourfold.

\begin{thmA}
    Let $X=V(\sum_{i=0}^5x_i^3)\subset\PP^5$ be the Fermat cubic. The locus of second type lines $F_2\subset F$ is
    isomorphic to the union of 10 copies of the self-product of the Fermat cubic curve and 45 copies of the Fermat cubic
    surface.
\end{thmA}

A salient feature of $F_2(X)$ in the case $n=4$ is that it is contained in the locus of indeterminacy of the Voisin map
 
\begin{align*}
    \phi:F(X)&\dashrightarrow F(X) \\
    \ell & \mapsto \ell' 
\end{align*}
where there exists a $\PP^2\cong\Pi\subset\PP^5$ such that $X\cap\Pi=2\ell+\ell'$. In fact, the indeterminacy locus is
equal $F_2(X)$ if $X$ contains no 2-planes, and $\phi$ is resolved by a single blowup.  We obtain the following (see
Proposition \ref{prop:thmB}) as an application of the above analysis by degenerating to the Fermat cubic fourfold.

\begin{thmB}
        Let $X\subset\PP^5$ be a general cubic fourfold and $E_{F_2}\subset \mathrm{Bl}_{F_2} F$ the exceptional
        divisor. Then the Voisin map $\phi: E_{F_2}\to F$ is birational onto its image, a divisor $D\subset F$.
\end{thmB}

The image of $E_{F_2}$ in the second part of the above statement has received considerable attention in recent years as
it provides an explicit uniruled divisor in the hyperk\"ahler fourfold $F$ (see \cite{osy,charlesmongardipacienza}). 

In a similar vein, either by degenerating to the Fermat if $n=3,4$, or more generally for any $n\geq3$ using a different
degeneration (cf.\ \cite[Remark 3.8]{cubicfourfolds1}), we present a proof of the fact that the natural morphism
$\bL_2\to X$ is birational onto its image, a divisor $W\subset X$. As an application of this and some more refined
statements in the case $n=3$, we compute (see Section \ref{sec:monodromy3}) that the monodromy groups of various natural
finite morphisms are the full symmetric group, where in the following $C_\ell\subset F(X)$ is the locus of lines which
meet $\ell$.

\begin{thmC}
    Let $X\subset\PP^4$ be a general smooth cubic threefold.
    \begin{enumerate}
        \item The degree 6 morphism $\bL\to X$ has monodromy group $S_6$.
        \item For $\ell\subset X$ a general line, the monodromy of the natural degree 5 morphism
        $\pi_\ell:C_\ell\to\ell, [\ell']\mapsto \ell\cap\ell'$ is $S_5$.
        \item For $\ell\subset X$ a general line of second type, the monodromy of the natural degree 4 morphism
        $\pi_\ell:C_\ell\to\ell, [\ell']\mapsto \ell\cap\ell'$ is $S_4$.
    \end{enumerate}
\end{thmC}

\textbf{Acknowledgements} We would like to thank Daniel Huybrechts for suggesting studying the monodromy problems for
cubic threefolds presented in Section \ref{sec:monodromy3}.

\section{Preliminaries}\label{sec:prelim}

Let $n\geq3$ and let $X\subset\PP^{n+1}$ be a smooth cubic $n$-fold, $F\subset\G(2,n+2)$ its Fano scheme of lines and
$F_2\subset F$ the locus of lines of second type.

Consider the induced morphism from the family of lines over $F$ and $F_2$ respectively 
\[
    \xymatrix{
        \bL_2\ar[d]\ar@{^{(}->}[r] & \bL\ar[r]^p\ar[d]^q & X \\
        F_2\ar@{^{(}->}[r]& F. &
}\]
for $\bL=\PP(\cU|_F)=\mathrm{Proj}(\Sym(\cU^\vee|_F))$ (i.e., one dimensional subspaces) where $\cU$ the universal bundle on
the Grassmannian $\G(2,n+2)$ and $\bL_2$ its restriction to $F_2$. We denote the scheme-theoretic image of the
morphism $\bL_2=q^{-1}(F_2)\to X$ by $W$. In other words, the locus $W$ is that spanned by second type lines. Recall
the following result, essentially going back to Clemens--Griffiths.

\begin{proposition}(\cite[\S 2.2]{huybrechts})\label{prop:Wdiv}
    Let $X\subset\PP^{n+1}$ be a smooth cubic hypersurface with $n>2$ and $F,\bL,\bL_2, W$ as above. Then 
    \begin{enumerate}
        \item $F_2$ is pure $n-2$-dimensional and reduced, and $W$ is reduced.
        \item The morphism $\bL_2\to W$ is generically finite and $W$ is a divisor, in particular of pure codimension one. 
        \item $\bL_2$ is the non-smooth locus of the morphism $p$.
        \item If $X$ is general, $F_2$ is smooth and irreducible and hence $W$ is irreducible. \label{prop:Wdiv4}
    \end{enumerate}
\end{proposition}
\begin{proof}
    These are a combination of \cite[2.2.11-2.2.15]{huybrechts}, noting that the description of $F_2$ as a degeneracy
    locus of the right codimension implies that it is Cohen--Macaulay, hence has no embedded points.
\end{proof}

In low dimensions, we have the following. In the cubic threefold case, the morphism $p$ is generically finite of degree
six and is ramified at second type lines, meaning they count with multiplicity at least two out of the six. For cubic
fourfolds, the generic fibre $C_x=p^{-1}(x)$ is a $(2,3)$ curve in $\PP^3$ (many of whose properties are known, see
\cite[\S 3]{cubicfourfolds1} for a summary), and the locus $\bL_2$ consists of the locus of singularities of the $C_x$.

If $n=4$ we use special notation for the above, in line with \cite[\S 2]{cubicfourfolds1}. From now on $X$ will be a
cubic fourfold and we denote by $S\subset F$ the locus of second type lines in its Fano scheme. Denote by
$\phi:F\dashrightarrow F$ the Voisin map, taking a general line $\ell$ to its residual intersection on the unique
tangent 2-plane to $\ell$. This map is not defined on the locus of lines of $X$ contained inside a 2-plane inside $X$
and on the second type locus $S\subset F$. If $X$ is general then it contains no 2-planes and $S$ is smooth and
irreducible, and we can resolve this map by a single blowup of $F$ at $S$ and we obtain a lift
$\phi:\mathrm{Bl}_S(F)\to F$. In this case the fibre $E_\ell\cong\PP^1$ of the exceptional divisor $E\subset\tilde{F}$
over a point $[\ell]\in S$, parametrises pairs $(\ell,\Pi)$ where $\Pi$ is a 2-plane tangent along the second type line
$\ell$ and in fact \cite[Remark 2.2.18]{huybrechts} gives that $\mathrm{Bl}_SF\cong\tF$ where
\[\tF=\{(\ell,\Pi):\Pi\text{ tangent along }\ell\}\subset\G(2,6)\times\G(3,6).\]

We also constructed in \cite[Proposition 4.2]{cubicfourfolds1} a universal family
\begin{equation}\label{eq:tF}
\tilde{\cF}:=\{(X,(\ell,\ell',\Pi)) : X\cap\Pi=2\ell+\ell'\text{ or }\ell,\ell'\subset\Pi\subset X\}\subset|\OO_{\PP^5}(3)| \times \cG
\end{equation}
where $\cG=\{(\ell,\ell',\Pi):\ell,\ell'\subset\Pi\}\subset\G(2,6)\times\G(2,6)\times\G(3,6)$ (actually, the case
$\ell,\ell'\subset\Pi\subset X$ was omitted from \cite{cubicfourfolds1} because the interest there was in general cubics
which do not contain any planes). From the discussion
above, if $X$ is general, then the fibre $\tilde{\cF}_X$ of the projection $\mathrm{pr}_1:\tilde{\cF} \to
|\OO_{\PP^5}(3)|$ over $X$ is just $\tilde{F}\cong\mathrm{Bl}_S(F)$. On the other hand, over special smooth cubics $X$,
for example the Fermat, $\tilde{\cF}_X$ is at least irreducible and birational to $F$ as seen from the following lemma.
We will mostly call the above locus $\tilde{\cF}_X$ to avoid confusing it with $\mathrm{Bl}_SF$.
\begin{lemma}\label{lem:tFXirred}
    Let $X\subset\PP^5$ be a smooth cubic fourfold. Then $\tilde{\cF}_X$ is irreducible and birational to $F$ the Fano
    variety of lines on $X$.
\end{lemma}
\begin{proof}
    Note that $\tilde{\cF}$ is smooth and irreducible from \cite[Proposition 4.2]{cubicfourfolds1} and for a general
    cubic $X$ the fibre $\tilde{\cF}_X$ is irreducible, smooth and birational to $F$ as discussed above. From semi-continuity of fibre
    dimension on the source \cite[\href{https://stacks.math.columbia.edu/tag/02FZ}{Lemma 02FZ}]{stacks-project}, for any
    smooth cubic $X$, every irreducible component of $\tilde{\cF}_X$ has dimension at least $4$. As $F$ is irreducible
    and 4-dimensional for any smooth cubic and by definition the fibres of $\pi:\tilde{\cF}_X\to F$ parametrise planes
    tangent to a line, they are generically one point, whereas for second type lines is 1-dimensional from
    \cite[Corollary 2.2.6]{huybrechts}. Hence if the locus in $\tilde{\cF}_X$ over $S\subset F$ were an irreducible
    component of $\tilde{\cF}_X$, it would be of dimension three, contradicting the above mentioned semi-continuity. In
    other words, for every smooth cubic $\pi$ is birational and $\tilde{\cF}_X$ is irreducible.
\end{proof}
For $Y\subset F$ a subvariety, we denote by
$E_Y\subseteq\tilde{\cF}_X$ its inverse image in $\tilde{\cF}_X$. From its very definition, the Voisin map extends to a
morphism $\phi=\mathrm{pr}_1:\tilde{\cF}_X\to F$ for any smooth cubic fourfold, by sending $\ell\mapsto\ell'$ in
Equation \eqref{eq:tF}. Note also that $\tilde{\cF}$ admits a
morphism to the universal Fano variety of lines $\cF$ by projecting to the first factor of $\cG$. In particular, we let
$\cE_\cS$ be the inverse image of the universal family of second type loci, which as noted, for general
$X\in|\OO_{\PP^5}(3)|$ is just a $\PP^1$-bundle over $S$, namely $\cE_\cS|_X\cong E$.

\section{The Fermat Cubic and Applications}\label{sec:fermat}

Before we specialise to the case of the Fermat cubic fourfold, we recall some facts about the threefold case (see
\cite{cg, roulleauelliptic, irrgk}). Let $Y\subset\PP^4$ be any smooth cubic threefold. Its Fano surface
$F(Y)\subset\rm{G}(2,5)$ is a smooth irreducible projective general type surface canonically embedded and of degree
$K_{F(Y)}^2=45$ in the Pl\"ucker $\PP^9$. The locus of second type lines $F_2(Y)\subset F(Y)$ is pure 1-dimensional and
of class $2K_{F(Y)}=\OO_{\PP^9}(2)|_{F(Y)}$.

\begin{proposition}\label{prop:fermat3}
    Let $Y=V(\sum_{i=0}^4 x_i^3)\subset\PP^4$ be the Fermat threefold. Then the locus of second type lines $F_2(Y)$
    consists of the union of 30 smooth elliptic curves, each isomorphic to the Fermat plane cubic. In particular,
    for $W=p(q^{-1}(F_2(Y)))\subset Y$, we have $\OO_Y(W)\cong\OO_Y(30)$ and that \[p:q^{-1}(F_2(Y))\to
    W\] is birational when restricted to any of the 30 irreducible components of $q^{-1}(F_2(Y))$.
\end{proposition}
\begin{proof}
    This basically follows from the work of Roulleau on the Fermat cubic \cite[Proposition 10 and
    p.395]{roulleauelliptic}. First note that from loc.\ cit.\ a smooth elliptic curve in $F(Y)$ corresponds to a cone
    in $X$ over a plane cubic, with vertex an Eckardt point of $X$. Letting $H=\OO_{\PP^9}(1)|_{F(Y)}\cong
    K_{F(Y)}\in\Pic F(Y)$ be the Pl\"ucker polarisation, one proves that each such elliptic curve $E$ has degree
    $H.E=3$. Now, loc.\ cit.\ constructs 30 distinct elliptic curves $E_i$ in $F(Y)$, each isomorphic to the Fermat
    plane cubic and their configuration is described in \cite[p.395]{roulleauelliptic}. As every line through an Eckardt
    point is of second type (see the beginning of Section \ref{sec:monodromy3}), we obtain $E_i\subset F_2(Y)$ for all
    $i=1,\ldots, 30$. Moreover, as $H.[F_2(Y)]=K_{F(Y)}.(2K_{F(Y)})=90$, it must be that $F_2(Y)=\sum E_i$ and so $W$ is
    the union of these 30 cones over plane Fermat cubics, each of which is a hyperplane section of $Y$, giving
    $\OO_Y(W)\cong\OO_Y(30)$. Finally
    \begin{align*}
    [\OO_Y(1)]^2p_*q^*[F_2(Y)]&=(q_*p^*[\OO_Y(1)]^2).[F_2(Y)]\\
        &=H.[F_2(Y)]\\
        &=90, 
    \end{align*}
    which forces the restriction of $p$ to any irreducible component $q^{-1}(E_i)$ of the universal second type family
    to be birational onto its image.
\end{proof}

Consider now $X=V(\sum_{i=0}^5 x_i^3)\subset\PP^5$ the Fermat cubic fourfold. For $\{i,j,k\}\subset\{0,\ldots,5\}$ consider the
hypersurfaces
\begin{align*}
    V_{i,j,k}&=V(x_i^3+x_j^3+x_k^3)\subset X\\
    V_{i,j}&=V(x_i^3+x_j^3)\subset X.
\end{align*}
If $\{i,j,k,i',j',k'\}=\{0,\ldots,5\}$ then $V_{i,j,k} = V_{i',j',k'}$, so in this way we obtain 10 distinct
hypersurfaces $V_{i,j,k}$ in $X$. Denote by
\begin{align}\label{eq:Cijk}
\begin{split}
C_{i,j,k}&=V_{i,j,k}\cap V(x_{i'},x_{j'},x_{k'})\\
C'_{i,j,k}&=V_{i',j',k'}\cap V(x_i,x_j,x_k). 
\end{split}
\end{align}
Then \[V_{i,j,k} =\mathrm{Join}(C_{i,j,k},C'_{i,j,k}),\] the join variety of $C_{i,j,k},C'_{i,j,k}$. 

For the $V_{i,j}$, note that there are 15 distinct hypersurfaces $V_{i,j}$ in $X$. Suppose that $0\leq i'<i''<j'<j''\leq 5$,
with $\{i',i'',j',j'',i,j\}=\{0,\ldots,5\}$ and let 
\begin{align}\label{eq:Gij}
G_{i,j}=V(x_{i'}^3+ x_{i''}^3+x_{j'}^3+x_{j''}^3, x_i,x_j)
\end{align}
a Fermat cubic surface. We have $V(x_i^3+x_j^3, x_{i'}, x_{i''}, x_{j'}, x_{j''})=\{p_{i,j}^{(1)},p^{(2)}_{i,j},
p^{(3)}_{i,j}\}$ corresponding to the three cube roots of $-1$. Then \[V_{i,j}=\mathrm{Join}(G_{i,j},
\{p^{(1)}_{i,j},p^{(2)}_{i,j}, p^{(3)}_{i,j}\})\] the join variety of $G_{i,j}$ with the set of points
$\{p^{(1)}_{i,j},p^{(2)}_{i,j}, p^{(3)}_{i,j}\}$. 

\begin{remark}
Actually the 45 points $p_{i,j}^{(\mu)}$ are all the Eckardt points of the Fermat cubic fourfold. They are Eckardt since
there is a 2-dimensional family of lines in $X$ through them, namely the lines joining the point $p_{i,j}^{(\mu)}$ with
the points of the surface $G_{i,j}$. Finally the maximal number of Eckardt points for a cubic fourfold is given as 45 by
\cite[3.12]{ccs}.
\end{remark}

\subsection{The second type locus for the Fermat cubic fourfold}

Recall first the following Lemma. 
\begin{lemma}(\cite[Definition 6.6]{cg})
    A line $\ell\subset X=V(f)$ in a smooth cubic fourfold is of second type if and only if the Gauss map
    $\Phi:X\to(\PP^5)^\vee$, $\Phi(x)=\nabla(f)(x)$ restricted to $\ell$ is a two-to-one covering of $\PP^1$. In other
    words, for every point $x\in \ell$, there is an antipodal point $x'\in\ell$ at which $\Phi(x)=\Phi(x')$.
\end{lemma}

In the case of the Fermat $X\subset\PP^5$, $\nabla(f)=\langle x_0^2, \ldots, x_5^2\rangle$ which we denote by $\nabla X$.
The locus $W\subset X$ spanned by second type lines is divisorial from Proposition \ref{prop:Wdiv} for any smooth
cubic.

Let $F\subset\mathrm{G}(2,6)$ be the Fano scheme of lines of the Fermat cubic and $S\subset F$ the second type locus. We
denote by $S_{i,j,k}\subset F$ the locus of ruling lines of the variety $\mathrm{Join}(C_{i,j,k},C'_{i,j,k})$ and by
$S^{(\mu)}_{i,j}\subset F$ the locus of ruling lines of the variety $ \mathrm{Join}(G_{i,j}, \{p^{(\mu)}_{i,j}\})$,
$\mu=1,2,3$. Note that $S_{i,j}^{(\mu)}$ is isomorphic to $G_{i,j}$ of Equation \eqref{eq:Gij}, i.e., to the Fermat
cubic surface. On the other hand, each $S_{i,j,k}$ is isomorphic to $C_{i,j,k}\times C'_{i,j,k}$ for two Fermat cubic
curves, since these two curves are disjoint.

\begin{lemma}
    We have that $V_{i,j,k}, V_{i,j}\subset W$ and $S^{(\mu)}_{i,j}, S_{i,j,k}\subset S$ for all
    $\{i,j,k\}\subset\{0,\ldots,5\}$ and $\mu\in\{1,2,3\}$.
\end{lemma}
\begin{proof}
    For $V_{i,j,k}$, by symmetry we may assume without loss of generality that $i=0,j=1,k=2$. Consider a general point
    \[p=[a_0:\cdots:a_5] \in V_{0,1,2}=V(x_0^3+x_1^3+x_2^3, x_3^3+x_4^3+x_5^3)\subset\PP^5,\] in the sense that it
    satisfies $a_0^3+a_1^3+a_2^3=a_3^3+a_4^3+a_5^3=0$ so that not all $a_i$ for
    $i\in\{0,1,2\}$ are zero nor $a_i=0$ for all $i\in\{3,4,5\}$. To prove that $p\in W$, we need to find a second type
    line $\ell_p$ containing $p$. For this, take $p'=[-a_0:-a_1:-a_2:a_3:a_4:a_5]$ (necessarily $\neq p$) and let
    $\ell_p$ be given parametrically by
    \[[(\lambda-t)a_0:(\lambda-t)a_1:(\lambda-t)a_2:(\lambda+t)a_3:(\lambda+t)a_4:(\lambda+t)a_5]\subset V_{i,j,k},\]
    i.e., the line $\lambda p + tp'$. Indeed $\ell_p$ is a line of second type as the antipodal (under the Gauss map) to
    the point $\lambda p+tp'$ is $tp+\lambda p'$, i.e., they have the same image under $\Phi$. The two ramification
    points of the Gauss map restricted to $\ell_p$ are $[a_0:a_1:a_2:0:0:0]$ and $[0:0:0:a_3:a_4:a_5]$.

    For $V_{i,j}$, again by symmetry we may consider only the case $i=0, j=1$. Consider a general point \[p=[a:\omega
    a:a_2:\cdots:a_5]\in V_{0,1} = V(x_0^3+x_1^3, x_2^3+\ldots+ x_5^3)\] for $\omega$ a cube root of $-1$ satisfying that
    neither $a=0$ nor $a_2=\ldots=a_5=0$. To find a second type line $\ell_p$ through $p$, we take $p'=[-a:-\omega
    a:a_2:\ldots:a_5]$ and define $\ell_p$ parametrically by $\lambda p + tp'$. One checks that $\ell_p\subset V_{0,1}$
    and that the antipodal point of $\lambda p+tp'$ is $tp+\lambda p'$ and that the ramification points of the Gauss map
    are $[0:0:a_2:\ldots:a_5]$ and $[a:\omega a:0:\ldots:0]$. In this case we end up with three irreducible components as
    $V_{0,1}$ is reducible.
\end{proof}

\begin{proposition}\label{prop:fermatdegW}
    We have 
    \begin{align*}
        W&=\cup_{i,j,k} V_{i,j,k}\bigcup\cup_{i,j} V_{i,j},\\
        S&=\cup_{i,j,k} S_{i,j,k}\bigcup\cup_{i,j,\mu} S_{i,j}^{(\mu)}.
    \end{align*}
    Moreover, the morphism $p_{q^{-1}(S)}:q^{-1}(S)\to W$ is generically of degree one when restricted to any one of the
    $55=3\cdot15+10$ irreducible components of $S$ as above.
\end{proposition}
\begin{proof}
    By definition $W=p(q^{-1}(S))$ and one computes $(p_*q^*[S]).H_X^3=225$ from \cite[Lemma 2.6]{cubicfourfolds1}, so
    that $\OO_X(W)=\OO_X(75)$. At this point it is not clear that the degree of $p:q^{-1}(S)\to W$ is generically one.
    On the other hand, we have constructed $10+15=25$ divisors $V_{i,j}, V_{i,j,k}$ above, all of which are contained in
    $W$. As their total degree is $3\cdot3\cdot25=225$, this forces their union to be $W$. Since $S$ is pure
    2-dimensional and the universal family over every irreducible component of it maps generically finitely to $X$ (see
    \cite[Lemma 2.2.12]{huybrechts}), the above also concludes that \[S=\cup_{i,j,\mu} S_{i,j}^{(\mu)} \bigcup \cup_{i,j,k}
    S_{i,j,k}.\]
    The degree of the map $p$ on any irreducible component of $q^{-1}(S)$ must also be generically one by degree
    considerations.
\end{proof}

\subsection{The Voisin map for the Fermat cubic fourfold}

In this section we denote by $\tF:=\tilde{\cF}_X$ for $X$ the Fermat cubic (see Section \ref{sec:prelim}), which is
irreducible and birational to $F$ from Lemma \ref{lem:tFXirred}. Note that as $S$ is singular (non-normal even), $\tF$
is not isomorphic to the blowup of $F$ at $S$, but is rather defined in terms of tangent 2-planes.  We recall also that
for $Y\subset F$ a subvariety, we denote by $E_Y\subseteq\tilde{\cF}$ its inverse image in $\tilde{\cF}$, so that for
example for $S_{i,j}^{(\mu)}, S_{i,j,k}$ the irreducible components of $S$ for the Fermat, defined in the previous
section, denote by $E_{S_{i,j}^{(\mu)}}, E_{S_{i,j,k}}\subset E:=E_S$ their inverse images in $\tilde{\cF}$.

\vspace{9pt}
\begin{figure}[ht]\label{fig1}
\centering
\begin{tikzpicture}
\node [
    above right,
    inner sep=0] (image) at (0,0) {\includegraphics[width=0.3\textwidth]{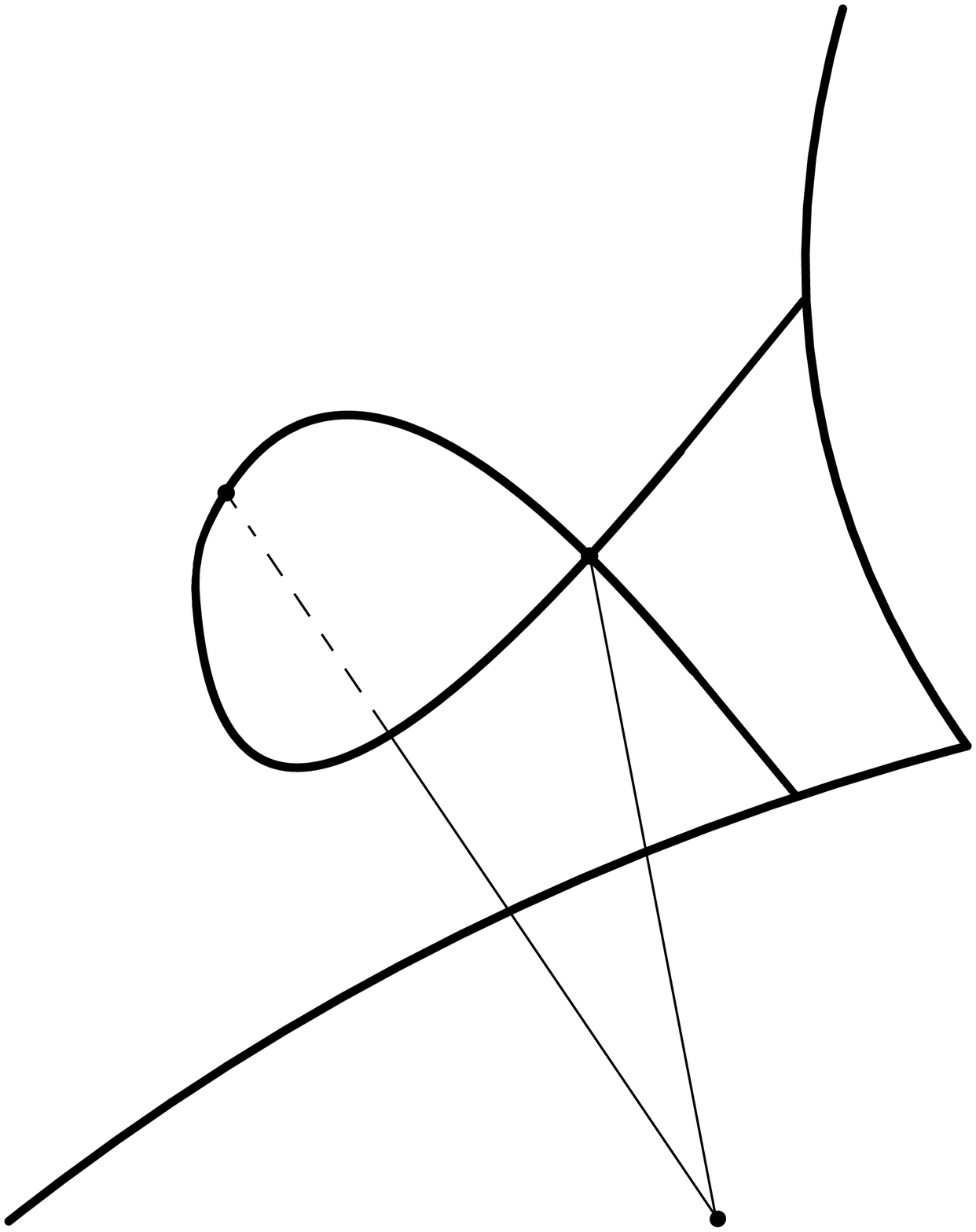}};
\begin{scope}[
    x={($0.1*(image.south east)$)},
    y={($0.1*(image.north west)$)}]

    \node[above left] at (2.2,6){$q$};
    \node[above left] at (7.4,6.7){$R_p$};
    \node[above] at (6.1,5.7){$p$};
    \node[right] at (7,2){$\ell$};
    \node[below left] at (6,1.5){$\ell'$};
    \node[right] at (9,6){$G_{i,j}$};
    \node[right] at (7.4,0.2){$p_0$};
 \end{scope}
\end{tikzpicture}
\caption{}
\end{figure}

\begin{lemma}\label{lem:imagephiSij} 
    The morphism $\phi$ contracts the threefold $E_{S_{i,j}^{(\mu)}}$ onto the surface $S_{i,j}^{(\mu)}$ in $F$. 
\end{lemma}
\begin{proof}
    We refer to Figure 1. Let $[\ell]\in S^{(\mu)}_{i,j}$ be a general point with $\ell = \overline{p_0p}$,
    $p_0=p^{(\mu)}_{i,j}$, $p\in G_{i,j}$. Then $R_{p} =T_{p}X\cap G_{i,j}$ is a cubic curve with a nodal singularity at
    $p$. As $\mathrm{Join}(p_0, G_{i,j})\subset T_{p_0}X$ and $\ell\subset T_pX$, we obtain \[\mathrm{Join}(p_0,
    R_{p})\subset T_{p}X\cap T_{p_0}X=H\cong {\mathbb P}^3.\] Then $R_{p}$ parametrises 2-planes in $H$ which contain
    $\ell$. Indeed, each such plane $\Pi$ intersects $R_{p}$ at a unique point $q$ (other that $p$) and then $\phi(\ell,
    \Pi)=\overline{p_0q}$. Therefore \[\phi(E_{S^{(\mu)}_{i,j}})\subset S^{(\mu)}_{i,j}.\] In fact it is easy to see that
    equality holds, since the smooth points of the curves $R_p$ for $p\in G_{i,j}$ cover the whole $G_{i,j}$.
\end{proof}

\vspace{10pt}
\begin{figure}[ht]\label{fig2}
\centering
\begin{tikzpicture}
\node [
    above right,
    inner sep=0] (image) at (0,0) {\includegraphics[trim={0 0 0 1cm},clip,width=0.55\textwidth]{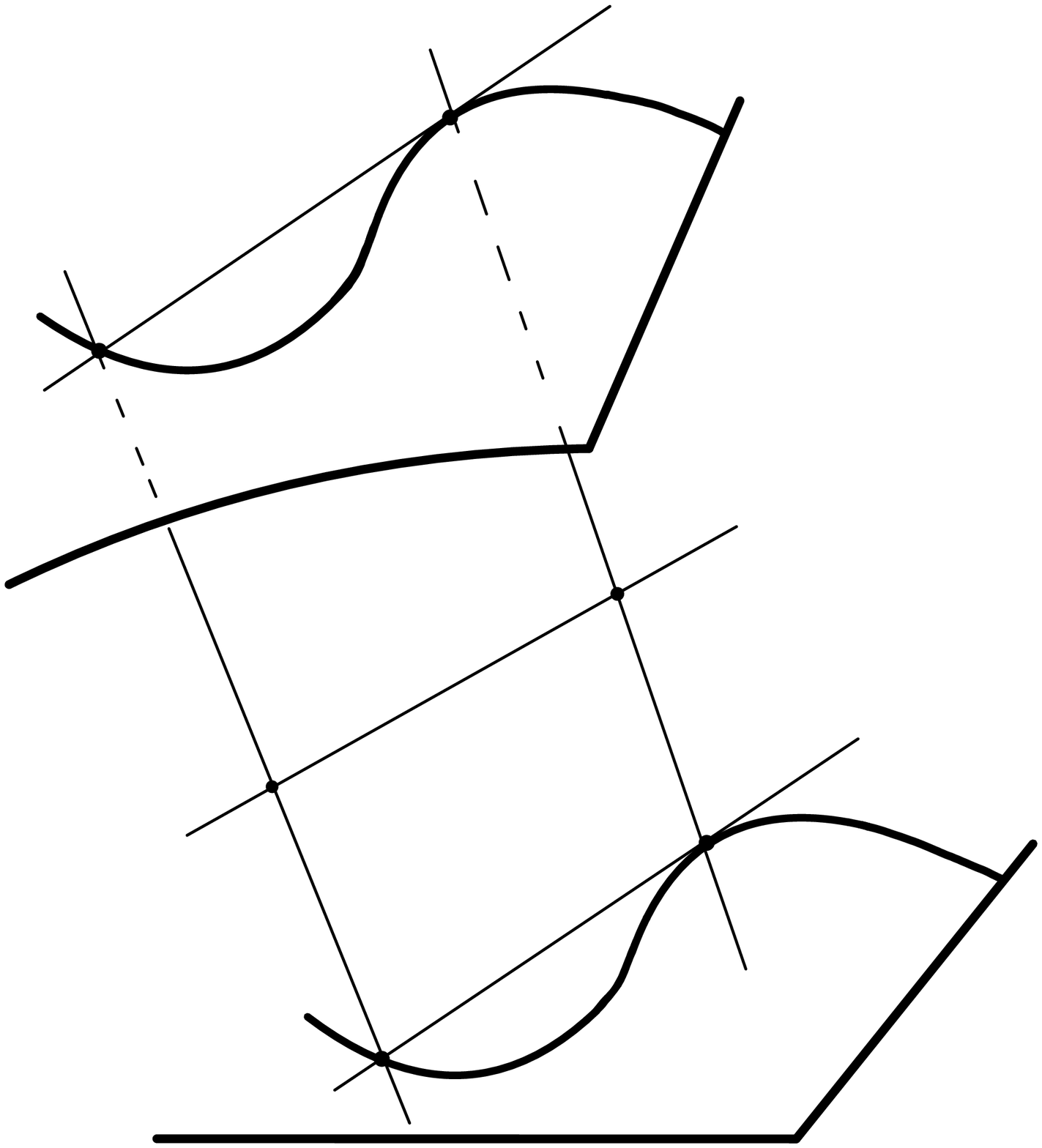}};
\begin{scope}[
    x={($0.1*(image.south east)$)},
    y={($0.1*(image.north west)$)}]

    \node[above] at (1.1,7.4){$p'$};
    \node[above] at (4.5,9.6){$p$};
    \node[above left] at (2.6,8.2){$T_pC_1$};
    \node[below] at (6,9.5){$C_1$};
    \node[right] at (6.2,7){$\Pi_1$};
    \node[right] at (8.4,0.6){$\Pi_2$};
    \node[right] at (7,5.2){$\ell_A'$};
    \node[above] at (6.1,5.2){$B$};
    \node[left] at (2.5,3.4){$A$};
    \node[left] at (2.9,2){$\tilde{\ell}$};
    \node[left] at (6.2,3.8){$\ell$};
    \node[above] at (6.9,3){$q$};
    \node[above] at (3.9,1){$q'$};
    \node[above] at (5,1.9){$T_qC_2$};
    \node[below] at (8.2,2.8){$C_2$};
 \end{scope}
\end{tikzpicture}
\caption{}
\end{figure}

We now discuss the restriction of the map $\phi$ to $S_{i,j,k}$, where we refer to Figure 2. Without loss of generality
we may assume that $i=0, j=1, k=2$. Let 
\begin{align*}
\Pi_1&=V(x_3, x_4, x_5)\cong {\mathbb P}^2_{(x_0,x_1,x_2)}\\
\Pi_2&=V(x_0, x_1, x_2)\cong {\mathbb P}^2_{(x_3,x_4,x_5)}. 
\end{align*}
A $p\in {\mathbb P}^5\backslash \Pi_1\cup \Pi_2$ can be written as
$p=[x: y]$ with $[x]\in \Pi_1$ and $[y]\in \Pi_2$. 
Geometrically, this means that $p$ is in the unique line joining $[x]\in \Pi_1$ with $[y]\in \Pi_2$ and also that the
ruling lines of $\mathrm{Join}(\Pi_1, \Pi_2)$ do not intersect at points outside $\Pi_1, \Pi_2$.
To simplify notation let 
\begin{align}\label{eq:Ci}
\begin{split}
C_1&=C_{0,1,2}=X\cap \Pi_1 = V(x_0^3+x_1^3+x_2^3)\subset\Pi_1\\
C_2&=C_{3,4,5}=X\cap \Pi_2 = V(x_3^3+x_4^3+x_5^3)\subset\Pi_2.
\end{split}
\end{align}
Note that $p=[x:y]\in \mathrm{Join}(C_1,C_2)$ (outside $C_1$ and $C_2$) if and only if $[x]\in C_1$ and $[y]\in C_2$. 

Let $p=[x_0:x_1:x_2]\in C_1$. We have 
\begin{align*}
\nabla_pX&=3\langle x_0^2,x_1^2,x_2^2,0,0,0\rangle\\
\nabla_pC_1&=3\langle x_0^2,x_1^2,x_2^2\rangle
\end{align*}
giving $T_pX|_{\Pi_1}=T_pC_1$. Similarly, for $q\in C_2$ we have $T_qX|_{\Pi_2}=T_qC_2$. 

Let now $\ell=\overline{pq}$, with $p, q$ general points of $C_1$, resp.\ $C_2$, a ruling line
of $\mathrm{Join}(C_1,C_2)$ as in Figure 2. The tangent line $T_pC_1$ intersects $C_1$ at an additional point $p'$ and,
similarly $T_qC_2$ intersects $C_2$ at an additional point $q'$. Since $p, q$ are general, $p'\neq p$ and $q'\neq q$. We
call the line $\tilde{\ell}=\overline{p'q'}$ the \textit{adjoint} line to $\ell$. We now have that $\ell, \overline{pq'},
T_pC_1\subset T_pX$ and $\ell, \overline{qp'}, T_qC_2\subset T_qX$. Then $p,q,p',q' \in T_pX\cap T_qX=H$ and since
$\ell, \tilde{\ell}$ are skew lines we have that $H\cong {\mathbb P}^3$ is spanned by the above points (note that
$T_pX\neq T_qX$ because the perp vectors at these two points are different). Then the points of the adjoint line
$\tilde{\ell}$ parametrise 2-planes in $H$ which contain $\ell$. For any $A\in \tilde{\ell}$, we denote by $\Pi_A$ the
plane spanned by $(p,q,A)$ and we have 
\begin{equation}\label{eq:ell'}
    \phi(\ell,\Pi_A)=\ell'_A,
\end{equation}
with $\ell'_A$ a line in $X$ which passes through $A$ and intersects $\ell$ at a point $B$. 

\begin{remark}\label{rem:pijflex}
    The point \[p^{(\mu)}_{i,j}\in C_{i,j,k}\] is a flex point and a ruling line $\ell$ of $\mathrm{Join}(C_1,C_2)$
    which passes through it is necessarily contained in $S_{i,j}^{(\mu)}$. Its image under $\phi$ is given as in Lemma
    \ref{lem:imagephiSij}. 
\end{remark}

\begin{lemma}\label{lem:imESij}
    For $0\leq i,j,i',j',k'\leq5$ distinct and $1\leq\mu\leq3$ we have \[\phi(E_{S^{(\mu)}_{i,j}}) \subset
    \phi(E_{S_{i',j',k'}}).\]
\end{lemma}
\begin{proof}
    Without loss of generality, we assume that $(i',j',k')=(0,1,2)$ and that $(i,j)=(4,5)$, $\mu=1$. Then
    $p_0:=p^{(1)}_{4,5}=[0:\cdots:0:1:-1]$. We will show that \[\phi(E_{S^{(0)}_{4,5}}) \subset \phi(E_{S_{0,1,2}}).\] Let
    $G=G_{4,5}$ and $C_1=C_{0,1,2}=\Pi_1\cap G$. As we have seen in Lemma \ref{lem:imagephiSij},
    \[\phi(E_{S^{(0)}_{4,5}})\subset S^{(0)}_{4,5}\] so it remains to show that $S^{(0)}_{4,5}\subset \phi(E_{S_{0,1,2}})$.
    For $[\ell'=\overline{p_0q}]\in S^{(0)}_{4,5}$ we have $q=[b_0:\cdots:b_3:0:0]\in G$, and let
    $q'=[b_0:b_1:b_2:0:0:0] \in \Pi_1$ be the point of intersection of $\ell'$ with $\Pi_1$. Let $p\in C_1$ with
    $\overline{q'p}=T_pC_1\subset \Pi_1$. Note that if $q=[x:y]\in G$ with $[x]\in T_pC_1$ then $q\in G\cap T_pX=R_p$:
    indeed, if $p=[a_0:a_1:a_2:0:0:0] $ then $\nabla_pX=3\langle a_0^2,a_1^2,a_2^2,0,0,0\rangle$ and therefore a point
    $q=[b_0:\cdots:b_3:0:0]\in G$ lies in $R_p$ if and only if $a_0^2b_0+a_1^2b_1+a_2^2b_2=0$, i.e., the point
    $[b_0:b_1:b_2:0:0:0]$ lies in $T_pC_1$, therefore $q'\in T_pC_1$. Note here that $C_1$ is contained in the Hessian
    of $G\subset\PP^3$ (given by $x_0x_1x_2x_3=0$) and so $R_p$ has a cuspidal singularity at $p$. By what we have
    described in Remark \ref{rem:pijflex}, if $\ell=\overline{p_0p}$ then $\phi (\ell, \Pi_q)=\ell'$ with $\Pi_q=\langle
    p_0pq\rangle$, which proves the claim.  Finally note that $[\ell]\in S_{0,1,2} \cap S^{(0)}_{4,5}$. 
\end{proof}

\begin{lemma}\label{lem:imSijk}
    The morphism $\phi$ is generically one-to-one when restricted to $E_{S_{i,j,k}}$.  Moreover, the images
    $\phi(E_{S_{i,j,k}})$ are all different from one another for different $S_{i,j,k}$.
\end{lemma}
\begin{proof}
    Given a general $\ell':=\ell'_A$ as in Equation \eqref{eq:ell'} (see also Figure 2), we first observe that in case
    $\ell'$ is not one of the lines which pass through  $p, p', q, q'$ then the only lines in $S_{i,j,k}$ which
    intersect $\ell' $ are the lines $\ell$ and its adjoint $\tilde{\ell}$: Indeed, if $[x:y]\in \ell'$ then $[x]\in
    T_pC_1$ and $[y]\in T_qC_2$, so since $T_pC_1$ intersects $C_1$ only at the points $p, p'$ and $T_qC_2$ intersects
    $C_2$ only at the points $q, q'$, the only other possible candidate line with image $\ell'$ under the map $\phi$ is
    the line $\tilde{\ell}$, in which case $T_pC_1$ would also be tangent to $C_1$ at $p'$, a contradiction since $C_1$
    is a cubic curve. 

    For the second claim, without loss of generality we show that $\phi(E_{S_{0,1,2}})\neq \phi(E_{S_{0,2,3}})$. We
    have seen in Lemma \ref{lem:imESij} that $S^{(0)}_{4,5}\subset \phi(E_{S_{0,1,2}})$. For a general $[\ell'] \in
    S^{(0)}_{4,5}$ we will show that $ [\ell']\notin \phi(E_{S_{0,2,3}})$. Indeed, $\ell'=[s,-s, ta_2,\cdots,ta_5]$,
    with $a_2^3+\ldots+a_5^3=0$. Then $\ell'$ intersects $V_{0,2,3}$ at the points $[s,t]$ satisfying the system
    $s^3+(ta_2)^3+(ta_3)^3=0$ and $(-s)^3+(ta_4)^3+(ta_5)^3=0$. By the above relation these two equations are equivalent
    and hence the intersection points correspond to $[s,t]$ with $(s/t)^3=a_4^3+a_5^3$. This has, generically, three
    distinct solutions. But from what we have seen for the first claim above, in order for $[\ell']$ to be in
    $\phi(E_{S_{0,1,2}})$ it has to intersect $V_{0,2,3}$ at two points only, or infinitely many. Hence the general such
    $[\ell']\notin \phi(E_{S_{0,2,3}})$, which proves the claim.
\end{proof}

In other words, the $E_{S_{i,j,k}}$ map generically one-to-one onto their images in $F$, whereas the
$E_{S_{i,j}^{(\mu)}}$ are contracted onto $S_{i,j}^{(\mu)}\subset\phi(E_{S_{i,j,k}})$. Hence, overall, we have a union
of divisors
\[\phi(E_S) =\bigcup_{i,j,k} \phi(E_{S_{i,j,k}}).\] 

\begin{proposition}\label{prop:360}
    Let $p\in X$ be a general point on the Fermat cubic fourfold. Then there are 360 distinct lines passing through $p$,
    each of which is a residual to some line of second type under the Voisin map.
\end{proposition}
\begin{proof}
    Let $p=[x:y]$ for $[x]\in \Pi_1$ and $[y]\in \Pi_2$ (with notation as in Equation \eqref{eq:Ci}). Then to each pair
    of a tangent line from $[x]$ to $C_1$ (there are 6 such choices) and a tangent line from $[y]$ to $C_2$ (also 6
    choices) we have that $[x:y]$ is contained in a line of $\phi(E_{S_{0,1,2}})$, giving 36 such choices of lines
    overall. Since we have 10 components $\phi(E_{S_{i,j,k}})$ the total number of lines in $\phi(E_S)$ which pass
    through $p$ are 360, noting that since $S^{(\mu)}_{i,j}$ is 2-dimensional, the span of such lines in $X$ is
    3-dimensional so there are no lines in these loci passing through a general point $p$.
\end{proof}

In the following $H\in\Pic F$ will denote the Pl\"ucker polarisation as usual.

\begin{corollary}\label{cor:phiESFermat}
    For $\phi:\tF=\tilde{\cF}_X\to F$ and $E_S\subset \tF$ on the Fermat $X$, we have that $[\phi(E_S)]=60H$ in $\Pic
    F$.
\end{corollary}
\begin{proof}
    Proposition \ref{prop:360} implies that for $x\in X$ a general point and $C_x=p^{-1}(x)$ the curve of lines through
    $x$, we have $[C_x].[q^{-1}(\phi(E_S))]=360$. Since $\phi_*[E_S]=60H\in\Pic F$ from \cite[Remark
    6.4.19]{huybrechts} and $[C_x].q^*[60H]=360$ too from \cite[Lemma 2.6]{cubicfourfolds1}. Note this also gives that the degree
    of the morphism $\phi|_{E_S}:E_S\to\phi(E_S)$ is one on the components $E_{S_{i,j,k}}$ which do not get contracted.
\end{proof}

To end this section, we address a question raised in \cite[Remark 6.4.19]{huybrechts}.

\begin{proposition}\label{prop:thmB}
    Let $X\subset\PP^5$ be a general cubic fourfold. The Voisin map $\phi:\tilde{\cF}_X\to F$ is generically one-to-one
    when restricted to $E_S\subset\tilde{\cF}_X$.  
\end{proposition}
In plain words, a general line which is residual to a second type line is residual to exactly one second type line.
\begin{proof}
    Let $\cE_\cS$ be the inverse image in $\tilde{\cF}$ (recall its definition in \eqref{eq:tF}) of the universal second
    type locus $\cS$. Denote by $\mathcal{D}=\phi(\cE_\cS)\subset\cF$ its scheme-theoretic image under the Voisin map,
    which comes with a morphism to $B=|\OO_{\PP^5}(3)|_{\mathrm{sm}}$. There is an open $U\subseteq B$, so that for
    $t\in U$, $\mathcal{D}_t$ is a divisor in $\cF_t$, and is reduced as $\mathcal{D}$ is so, and irreducible from
    Proposition \ref{prop:Wdiv}.\eqref{prop:Wdiv4}. Over points in $U$, $\mathcal{D}_t$ agrees with $\phi(E_{S_t})$.
    From the principle of conservation of numbers \cite[Proposition 10.2]{fulton}, since from Corollary
    \ref{cor:phiESFermat} we know that $[C_x].q^*[\phi(E_S)]=360$ on the general point $x$ of the Fermat, the same must
    be true of $[C_{x,t}].q^*[\mathcal{D}_t]$ for every point $t\in B$. As for $t$ general the class of $\phi(E_{S_t})$
    in $\Pic F_t$ is $60H$ (see Corollary \ref{cor:phiESFermat}) and $[C_x].q^*(60H)=360$, we obtain that for $t\in U$,
    $[\phi(E_{S_t})]=\phi_*[E_{S_t}]$ which implies that the degree of $\phi$ restricted to $E_{S_t}$ is one as
    required.
\end{proof}

\begin{remark}
    \begin{enumerate}
        \item As $K_F=0$ we obtain that for a general cubic fourfold, $K_{\tF}=E_S$. In particular the morphism
        $\phi:\tF\to F$ is ramified along $E_S$. The above shows that above a general branch point of $F$, there is
        precisely one ramification point. It would be interesting to know whether the ramification at this point is
        simple - in particular in view of computing the monodromy group of $\phi$ (cf.\ Section \ref{sec:monodromy3}).
        \item Presumably the $[C_x].[q^{-1}(\phi(E_S))]=360$ points of $C_x$ for $x\in X$ general are connected to the
        intersection points of the 120 tritangent 2-planes a space $(2,3)$-curve has.
    \end{enumerate}
\end{remark}

\section{Some Properties of the Second Type Locus}\label{sec:univsecondtype}

In previous work of ours (\cite[Remark 3.8]{cubicfourfolds1}) we sketched how to obtain the following
result, using the correspondence between hyperplane sections with two nodes and $(2,3)$-complete intersections (cf.\
\cite[Lemma 6.5]{cg}). At the time we had given a proof in the case of cubic fourfolds, but as this is an important
property of the second type locus and we will be using it for threefolds too, we include its proof for any $n\geq3$ here.

\begin{theorem}\label{thm:t2bir} 
    Let $n\geq3$ and let $X\subset\PP^{n+1}$ be a general cubic $n$-fold, $F$ its Fano scheme of lines and $F_2\subset
    F\subset \G(2,n+2)$ the locus of lines of second type. Consider the induced morphism from the family of lines 
    \[\xymatrix{
        \bL_2\ar[r]^-p\ar[d]^q & W\subset X. \\
        F_2 &
    }\]
    Then $p$ is birational onto its image $W$.
\end{theorem}
In other words on a general cubic, a general point through which passes a second type line has exactly one second
type line through it.
\begin{proof}
    First one notes that there is a bijection between pairs $(Y,x)$ where $Y\subset\PP^n$ is a cubic $(n-1)$-fold with
    two $A_1$ singularities, one of which is $x$, and $(2,3)$-complete intersections in $\PP^{n-1}$ with one $A_1$
    singularity and no other singular points (see \cite[Remark 3.8]{cubicfourfolds1}).

    Starting from such a $(2,3)$-complete intersection in $\PP^{n-1}$ \[C_0\cap Q_0=V(f_2)\cap V(f_3)\] and taking its
    corresponding $(Y_0,p_0)$ in $\PP^n$, with equation \[f=f_2\cdot x_n+f_3, \hspace{5pt} p_0=[0:\cdots:0:1],\] one
    constructs a smooth cubic $n$-fold $X_0\subset\PP^{n+1}$ (cf.\ the proof of \cite[Proposition 3.5]{cubicfourfolds1})
    with equation \[g=f + x_n^2x_{n+1}\] containing the point $x_0=[0:\cdots:0:1:0]$ whose tangent
    hyperplane meets $X_0$ at the pair $(Y_0, p_0)$. The span of lines through $x_0$ is a cone over $C_0\cap Q_0$. As
    for any smooth cubic $X$ the singular locus of $p:\bL\to X$ is precisely $\bL_2$ (see Proposition \ref{prop:Wdiv}),
    the above constructs one point in $X_0$ with a second type line through it corresponding to the unique singularity
    of $p^{-1}(x_0)=C_0\cap Q_0$. As having a unique $A_1$ singularity is an open condition, this must hold generically for
    (the necessarily singular) fibres over points of $p(\bL_2)=W_0\subset X_0$. This implies that for $X_0$, the map
    $\bL_2\to W_0$ is birational. 

    To conclude for the general cubic $n$-fold we argue as follows. Consider the universal smooth family
    $\cX\to|\OO_{\PP^{n+1}}(3)|_{\mathrm{sm}}$ and the universal locus $\cW:=p(q^{-1}(\mathcal{F}_2))$ over
    $|\OO_{\PP^{n+1}}(3)|_{\mathrm{sm}}$, where $\mathcal{F}_2$ is the universal second type locus, known to be
    irreducible from \cite[Proof of Proposition 2.2.13]{huybrechts}. From Proposition \ref{prop:Wdiv}, as the locus of
    second type lines $F_{2,t}$ of any smooth cubic $\cX_t$ is pure $(n-2)$-dimensional, we obtain that the
    $\PP^1$-bundle $q^{-1}(F_{2,t})$ is pure $(n-1)$-dimensional, and the same proposition gives that $\cW_t$ is a
    divisor in $\cX_t$ for any $t\in |\OO_{\PP^{n+1}}(3)|_{\mathrm{sm}}$. Hence the morphism $q^{-1}(F_{2,t})\to \cW_t$
    is generically finite. If $t$ is general then $F_{2,t}, \cW_t$ are irreducible and reduced. As over some points,
    namely the generic point of $W_0$ in the special cubic $X_0$ above we know that the fibre in $\bL$ has a single
    $A_1$ singularity, we obtain the same for the generic cubic. In particular, the morphism
    $p:q^{-1}(F_{2,t})\to\cW_t$ is birational for a general $t$ as required. 
    \end{proof}

\begin{remark}\label{rem:err}
    In \cite[Fact 3.2.(2)]{irrgk} we erroneously stated that in the case of a cubic threefold $X\subset\PP^4$,
    \cite[10.18]{cg} claims that $\bL\to X$ has one single ramification point of ramification index two over a generic
    point of the branch locus. What they do claim (and prove) is that for a generic point $x\in X$ through which there
    passes a line of second type, this line will count with multiplicity two as one of the six lines through $x$. In
    particular there could a priori still be multiple ramification points, i.e., other second types lines through $x$.
    The above proposition indeed confirms the stronger claim. Either way, this does not affect any results in
    \cite{irrgk} as in the proof we only required what the weaker statement about multiplicities pertains to. 
\end{remark}

One can derive another proof of Theorem \ref{thm:t2bir} for a general cubic threefold or fourfold by degenerating
to the Fermat (instead of $X_0$ used in the above proof) and using Propositions \ref{prop:fermat3},
\ref{prop:fermatdegW} respectively. In these cases the class of $W$ in the Picard group of the cubic is $\OO(30)$ and
$\OO(75)$ respectively.

We now extend the proof of Theorem \ref{thm:t2bir} to obtain the following refinement which will be used in the next
section.

\begin{lemma}\label{lem:ram3fold}
    Let $X\subset\PP^4$ be a general cubic $3$-fold and $\ell\in F_2$ a general second type line. Then there is a
    point on $\ell$ through which there pass precisely one further second type line and two more of first type.
\end{lemma}
\begin{proof}
   The same correspondence as in the proof of Theorem \ref{thm:t2bir} can be extended to
    \begin{enumerate}
        \item $(Y,x)$ where $Y\subset\PP^n$ is a cubic $(n-1)$-fold with three non-collinear ordinary double points, one
        of which is $x$, 
        \item $(2,3)$-complete intersections in $\PP^{n-1}$ with two ordinary double points.
    \end{enumerate}
    We recall the construction (see also \cite[Lemma 1.2]{amerikvoisin} for the fourfold case). The $(2,3)$-complete
    intersection $C\cap Q\subset\PP^{n-1}$ can be blown up to obtain a variety $\widetilde{Y}\to\PP^{n-1}$ which has two
    singularities above the singular points of $C\cap Q$. One now contracts the strict transform $\widetilde{Q}$ of $Q$
    to obtain $Y$, which is also singular at the image of $\widetilde{Q}$, the marked point. 

    The same method as we sketched in the proof of Theorem \ref{thm:t2bir} gives that starting with a
    $(2,3)$-intersection as above, we obtain a smooth cubic $X\subset\PP^{n+1}$ with a point $x_0$ whose tangent
    hyperplane meets $X$ at a $(Y,x_0)$ as in the correspondence. In particular, the two other singular points of $Y$ will
    correspond to two distinct second type lines through $x_0$, while every other line through $x_0$ will be of first type.

    Applying the above in the $n=3$ case we obtain a special smooth cubic, call it $X_0\subset\PP^4$ which has a second
    type line $\ell_0$ and a point $p_0\in\ell_0$ so that there are precisely four lines in $X_0$ passing through $p_0$,
    two of second type, and two of first. 

    Let $X\subset\PP^4$ be a smooth cubic fourfold. For a line $\ell\in F_2$, consider $C_\ell\subset F$ the locus of
    lines meeting $\ell$. We have a degree four morphism $C_\ell\to\ell$ taking a line $\ell'$ meeting $\ell$ to the point
    $\ell\cap\ell'$. Consider the universal family 
    \[
    \xymatrix{
        \mathcal{C}\ar[d]\ar[r]^-\pi & \bL_2=q^{-1}(\cF_2)\ar[dl]^{\mathrm{pr}_2} \\
        \cF_2
    }
    \]
    where for $\ell\in\cF_2$, the fibre in $\mathcal{C}$ over $\ell$ is $C_\ell$, and $\pi_\ell: \mathcal{C} \to
    \PP^1_\ell$ is the one described above. For any $\ell$, $\pi_\ell$ is ramified at the points of intersection of
    $C_\ell$ with $F_2$ in $F$, in particular at a non-empty locus. Furthermore, there exist $p_0\in
    \ell_0,\ell_0'\subset X_0$ both second type lines constructed in our special $X_0$ above where the ramification of
    $\pi_{\ell_0}:C_{\ell_0}\to \ell_0$ at the point $[\ell_0']\in C_\ell$ is simple. This ramification type must hold
    generically over $\cF_2$, for example since the ramification index $e_{\ell_t} = \operatorname{length}
    \big((\Omega_\pi)_{[\ell_t]}\big)+1$ is upper semicontinuous (see \cite[Exercise 7.1.6]{liu} and the upper
    semicontinuity of the rank) over the branch locus of $\pi$, which is a divisor from purity and
    smoothness of $\cF_2$.
\end{proof}

\section{Monodromy for Cubic 3-Folds}\label{sec:monodromy3}

In this section, let $X\subset\PP^4$ be a smooth cubic hypersurface, and denote by $F, F_2, \bL, \bL_2$ as in Section
\ref{sec:prelim}. We refer to \cite{cg} (cf. \cite[Remark 5.1.6]{huybrechts} for the following properties. In the
threefold case the morphism $p:\bL\to X$ is generically finite of degree six, i.e., there are six lines through a general
point of $X$. There are more than 6 lines through a point if and only if the point is an Eckardt point, in which case
there are infinitely many lines through that point, all of second type. A cubic threefold can contain at most 30 Eckardt
points (in fact this is the case for the Fermat, see Proposition \ref{prop:fermat3}), but the generic one contains none,
so in this case $\bL\to X$ is finite. For a line $\ell\in F$, we denote by $C_\ell\subset F$ the locus of lines meeting
$\ell$ - more precisely, $C_\ell$ is the image in $F$ under $q$ of the closure of $p^{-1}(\ell)\setminus
q^{-1}([\ell])$.

Given a generically finite dominant morphism $f:X\to Y$ of degree $d$ between irreducible varieties (necessarily of the same
dimension), we obtain a degree $d$ extension of function fields $k(X)/k(Y)$, and taking the Galois closure $K/k(Y)$ of
this extension, we denote by $\mathrm{Gal}_f=\mathrm{Gal}(K/k(Y))$. This agrees with the usual monodromy group (see
\cite{harrisgalois, sottileyahl}) $\mathrm{Mon}_f$, which is defined as the image in $S_d$ of the group of deck
transformations of the unramified (i.e., topological) cover $X\setminus\mathrm{ram}(f)\to U$, where $U =
Y\setminus\mathrm{branch}(f)$ is the largest dense open in $Y$ over which $f$ is \'etale.

Recall the following classical results for $f:X\to Y$ and $U$ as above.

\begin{lemma}(\cite[p698]{harrisgalois})\label{lem:harris}
    If there exists a fibre of $f$ with a unique point of ramification index two and all other points unramified, then
    $\mathrm{Mon}_f\subseteq S_d$ contains a transposition.
\end{lemma}

Let $X^{(s)}_U$ be the complement of the big diagonal in the fibre product of $X$ $s$-times with itself over $U$. In
other words, over a point of $U$, $X^{(s)}_U$ consists of $s$ distinct points in the fibre.

\begin{lemma}(\cite[Proposition 2]{sottileyahl})\label{lem:montrans}
    $X^{(s)}_U$ is irreducible if and only if $\mathrm{Mon}_f$ is an $s$-transitive subgroup of $S_d$.
\end{lemma}

We note the well-known fact that a subgroup of $S_d$ which contains a transposition and that is $2$-transitive must be
the whole $S_d$.

\begin{proposition}
    If $X$ is a general cubic threefold, the Galois group $G$ of $\bL \to X$ is $S_6$.
\end{proposition}
\begin{proof}
    Since $\bL=\PP(\cU|_F)$ is the projectivised universal bundle, it is irreducible and hence $G$ is a transitive subgroup
    of $S_6$ from Lemma \ref{lem:montrans}. If $G$ is 2-transitive and contains a transposition, it must be the whole $S_6$,
    so we now show each of these facts.

    First we show that $G$ contains a transposition. From Theorem \ref{thm:t2bir}, a general second type line
    $\ell\subset X$ counts with multiplicity two (out of the six) at a general point on it, and there are four other
    distinct type one lines through that point, i.e., there is a simply ramified fibre of $\bL\to X$, since from
    \ref{prop:Wdiv} the ramification locus of $\bL\to X$ is precisely the universal second type
    locus. Now Lemma \ref{lem:harris} implies that $G$ contains a transposition.

    Finally, we prove that $G$ is 2-transitive. Note that over a point $[\ell]$, the morphism $\bL\times_X \bL\to F$ has
    fibre $C_{\ell}$ the curve of lines meeting $\ell$. For any $X$, the general $C_{\ell}$ is a curve, and is
    irreducible as the class of $C_\ell$ is ample. Restricting to the open $U\subset X$ over which $\bL\to X$ is \'etale
    (note this also removes any Eckardt points, above which the fibre is infinite), we obtain a flat morphism
    $\bL^{(2)}_U\to F|_U$, since it a composition of the finite (hence flat) morphism $\bL^{(2)}_U\to\bL_U$ and the flat
    morphism $\bL|_U\to F|_U$. As the base and the general fibre are irreducible and the morphism is open, so too is the
    total space $\bL^{(2)}_U$ irreducible, from an elementary lemma
    \cite[\href{https://stacks.math.columbia.edu/tag/004Z}{Tag 004Z}]{stacks-project}. We conclude from Lemma
    \ref{lem:montrans}.  
\end{proof}

Consider next the natural morphism 
\begin{align*}
\pi_\ell:C_{\ell}&\to\ell\cong\PP^1 \\
[\ell'] &\mapsto\ell\cap\ell'.
\end{align*}
As for any smooth cubic and general line $\ell\subset X$ the variety $C_{\ell}$ is a curve, and there are six
lines through a point, the degree of this morphism is five. Analogously, 
if $\ell$ is a second type line, as it counts with multiplicity at least two as one of the six, we obtain an induced
degree 4 morphism $\pi_\ell:C_\ell\to\ell$.

\begin{proposition}\label{prop:gpcurve}
    Let $X$ be a general cubic threefold. If $\ell$ is a general line in $X$, then the $\mathrm{Mon}_{\pi_\ell}\cong
    S_5$. If $\ell$ a general line of second type, then $\mathrm{Mon}_{\pi_\ell}\cong S_4$.
\end{proposition}
\begin{proof}
    A general $\ell$ of first type will meet a general line of second type $\ell’$ at a general
    point $p$ of $\ell’$ (more specifically, outside the finite locus of points through which there pass more second type
    lines). This means from Theorem \ref{thm:t2bir} that $C_\ell\to \ell$ is simply ramified at
    $[\ell']$ over $p$ and there are no further ramification points, hence the Galois group has a transposition from
    Lemma \ref{lem:harris}. Transitivity in $S_5$ follows since $C_\ell$ is irreducible for general $\ell$ on any $X$.
    Since $S_5$ is the only subgroup of $S_5$ which has a transposition and is transitive, we conclude the result.

    To prove the second claim, note that Lemma \ref{lem:ram3fold} gives that there is always a transposition in the
    monodromy group. Again, as $C_\ell$ is irreducible from \cite[Lemma 3.3]{irrgk} and $S_4$ is the only transitive
    subgroup of $S_4$ which contains a transposition, we obtain the result.
\end{proof}


\end{document}